\documentclass[11pt]{amsart}
\usepackage{amsmath,amssymb,amsthm,amscd}
\usepackage{tabularx}
\newcommand{\Z}{\mathbb{Z}}
\newcommand{\Q}{\mathbb{Q}}
\newcommand{\ord}{{\rm ord}}
\usepackage{amsmath}
\usepackage{amssymb}
\usepackage{amssymb,amsmath,amsthm,amsfonts,latexsym,euscript}
\usepackage{wasysym}
\usepackage{fancyhdr}
\fancypagestyle{plain}%redefining plain pagestyle
\usepackage{hyperref}
\newtheorem*{theorem*}{Theorem}
\newtheorem{theorem}{Theorem}

\newtheorem{lemma}[theorem]{Lemma}
\newtheorem{corollary}[theorem]{Corollary}

\newtheorem*{remark}{Remark}
\newtheorem{proposition}[theorem]{Proposition}
\usepackage{amssymb}
\usepackage[T1]{fontenc}
\usepackage[latin1]{inputenc}
\usepackage[usenames,dvipsnames]{xcolor}
\usepackage{graphicx}
\usepackage{setspace}
\usepackage{txfonts}
\usepackage{longtable}
\usepackage{multirow}
\usepackage{blindtext}
\usepackage{ulem}
\usepackage{bm}
\usepackage[margin=1in]{geometry}% http://ctan.org/pkg/geometry

\newcommand{\SL}{{\rm SL}}

\AtEndDocument{{\footnotesize%
 \textsc{J.Rouse, Department of Mathematics, Wake Forest University, Winston-Salem, NC., 27109} \par  
  \textit{E-mail address}, J.~Rouse: \texttt{rouseja@wfu.edu} \par
  \addvspace{\medskipamount}
  \textsc{K. Thompson, Department of Mathematics, United States Naval Academy, Annapolis, MD., 21402} \par
  \textit{E-mail address}, K. ~Thompson: \texttt{kthomps@usna.edu} 
 }}

\title{Quaternary Quadratic Forms with Prime Discriminant}
\author[J.Rouse]{Jeremy Rouse}
\author[K.Thompson]{Katherine Thompson}
\begin{document}

%\tableofcontents

%%%%%%%%%%%%%%%%%%%%%%%%%%%%%%%%%%%%%%%%%
%% ABSTRACT
%%%%%%%%%%%%%%%%%%%%%%%%%%%%%%%%%%%%%%%%%

\begin{abstract}
  Let $Q$ be a positive-definite quaternary quadratic form with prime discriminant. We give an explicit lower bound on the number of representations of a positive integer $n$ by $Q$. This problem is connected with deriving an upper
  bound on the Petersson norm $\langle C, C \rangle$ of the cuspidal part of
  the theta series of $Q$. We derive an upper bound on $\langle C, C \rangle$ that depends on the smallest positive integer not represented by the dual form $Q^{*}$. In addition, we give a non-trivial upper bound on the sum of the integers $n$ excepted by $Q$. 
\end{abstract}
\maketitle

%%%%%%%%%%%%%%%%%%%%%%%%%%%%%%%%%%%%%%%%%%%%%%%%%%%
%%%%%%%%%%%%%%%%%%%%%%%%%%%%%%%%%%%%%%%%%%%%%%%%%%%
%%%%%%%%%%%%%%%%%%%%%%%%%%%%%%%%%%%%%%%%%%%%%%%%%%%
%%%%%%%%%%%%%%%%%%%%%%%%%%%%%%%%%%%%%%%%%%%%%%%%%%%
%%%%%%%%%%%%%%%%%%%%%%%%%%%%%%%%%%%%%%%%%%%%%%%%%%%
%%%%%%%%%%%%%%%%%%%%%%%%%%%%%%%%%%%%%%%%%%%%%%%%%%%
%%%%%%%%%%%%%%%%%%%%%%%%%%%%%%%%%%%%%%%%%%%%%%%%%%%

%%%%%%%%%%%%%%%%%%%%%%%%%%%%%%%%%%%%%%%%%
%% INTRODUCTION
%%%%%%%%%%%%%%%%%%%%%%%%%%%%%%%%%%%%%%%%%

\section{Introduction and Statement of Results}
%%%%%%%%
%%%%%%%%

The study of which integers are represented by a positive-definite quadratic form is an old one. After initial %classic and 
predominantly algebraic results from Euler, Lagrange and Legendre, in the 19th century beginning with Jacobi's formula for the number of representations of an integer by the sum of four squares, analytic techniques of increasing complexity have been applied. Towards the classification specifically of universal positive-definite quadratic forms (those which represent all $n \in \mathbb N$) were results from Ramanujan \cite{Ramanujan}, Dickson, and eventually the Conway-Schneeburger 15-Theorem \cite{Bhargava} and the Bhargava-Hanke 290-Theorem \cite{BH}. With almost universal forms (forms which are not universal and which fail to represent finitely-many $n \in \mathbb N$) or with forms which represent entire congruence classes are recent works respectively of Barowsky et. al \cite{Bar} and Rouse \cite{Rouse}. Additionally, in the development of the theory of modular forms (and in particular bounds on cusp forms) are results by Tartakowsky \cite{Tart}, Ono-Soundararajan \cite{OS}, and Schulze-Pillot \cite{SPY}. 

To give specific context to the results in this paper, we highlight a few predecessors. %First, we say that $Q$ locally represents a positive integer $n$ if the equation
%$Q(\vec{x}) = n$ has a solution with $\vec{x} \in \Z_{p}^{r}$ for all
%primes $p$. 
First, recall that a quadratic form $Q$ in $r$ variables is
anisotropic at a prime $p$ if the only solution to $Q(\vec{x}) = 0$
with $\vec{x} \in \Z_{p}^{r}$ is $\vec{x} = \vec{0}$. If $Q =
\frac{1}{2} \vec{x}\phantom{t}^{T} A \vec{x}$, where $A$ is an $r
\times r$ symmetric matrix with integer entries and even diagonal
entries, we let $D(Q) = \det(A)$ be the discriminant of $Q$. With
that, we begin with the following result:

\begin{theorem*}[Tartakowsky, 1929]
  Let $Q$ be a positive-definite quadratic form in $4$ variables and fix
  an integer $k \geq 0$. Then there is a constant $C(Q,k)$ such that if
  $n$ is a positive integer with $n \geq C(Q,k)$ and ${\rm ord}_{p}(n) \leq k$ for all anisotropic primes $p$, then $n$ is represented by $Q$.
\end{theorem*}

A number of more recent projects including \cite{BD} and \cite{SP} %(Fomenko, Schulze-Pillot 2001, Browning-Dietmann) 
 have
given bounds on the quantity $C(Q,k)$ for forms in four or more variables. In \cite{Rouse}, the first author proved the following result regarding forms with fundamental discriminant (which do not have any anisotropic primes):
\begin{theorem*}[Rouse, 2014]
If $D(Q)$ is a fundamental discriminant and $\epsilon > 0$, there is a constant $C_{\epsilon}$ so that all positive integers $n \geq C_{\epsilon} D(Q)^{2+\epsilon}$ are represented by $Q$.
\end{theorem*}
%In the previous result, no assumption about anisotropic primes is needed since
%a quadratic form with fundamental discriminant does not have any anisotropic primes.
%In the previous result, no assumption about anisotropic primes is needed since
%a quadratic form with fundamental discriminant does not have any anisotropic primes.

The above result cannot be made effective because of the possibility of a Dirichlet $L$-function $L(s,\chi)$ having a Siegel zero. The goal of the present
paper is to prove a stronger result and make it completely effective in the case that $D(Q) = p$ is a prime number. If $Q = \frac{1}{2} \vec{x}\phantom{t}^{T} A \vec{x}$
is such a form, the dual form $Q^{*} = \frac{1}{2} \vec{x}\phantom{t}^{T} pA^{-1} \vec{x}$
also has level $p$. Let $\min Q^{*}$ denote the smallest positive integer represented by $Q^{*}$. Also, let $M(n) = \max_{1 \leq m \leq n} \tau(m)$. Note
that $M(n) = O(n^{\epsilon})$ for all $\epsilon > 0$. More precisely, the bound
$M(n) \leq n^{\frac{1.538 \log(2)}{\log \log(n)}}$ for $n \geq 3$ is proven in \cite{NR}.

\begin{theorem}
\label{ineq}  
Let $Q$ be a positive-definite quaternary quadratic form with prime discriminant $p$. Let $n$ be a positive integer, and denote by $\phi(n)$ Euler's totient, and $\tau(n)$ the number of divisors of $n$. If $p \geq 101$, then
\[
  r_{Q}(n) \geq \frac{24}{p^{3/2}} (p-1) \phi(n) - 23.85 \sqrt{p \log(p) \left( \frac{1}{\min Q^{*}} + \frac{3216.66 M(25.09p^{35/6})}{p^{1/4}}\right)} \tau(n) n^{1/2}
\] where $r_Q(n) := \{ \vec{x} \in \mathbb Z^4 \vert Q(\vec{x}) = n\}$.
\end{theorem}
\begin{remark}
The inequality above shows that $r_{Q}(n) > 0$ provided that $n \gg \max \left\{ \frac{p^{2+\epsilon}}{\min Q^{*}}, p^{7/4+\epsilon} \right\}$.
\end{remark}

The method of proof is to take $\theta_{Q}(z) := \sum_{n=0}^{\infty} r_{Q}(n) q^{n}$, $q = e^{2 \pi i z}$ and decompose it as $\theta_{Q}(z) = E(z) + C(z)$
into an Eisenstein series and a cusp form. The cusp form in turn may be decomposed as
\[
C(z) = \sum_{i} c_{i} g_{i}(z)
\]
where each $g_{i}(z)$ is a newform. The Deligne bound on the Fourier coefficients of $g_{i}(z)$ implies that the $n$th coefficient is bounded above by $ \tau(n) \sqrt{n}$.
It suffices then to estimate the coefficients $c_{i}$. We do this by using the orthogonality of the Petersson inner product
and finding an upper bound on $\langle C, C \rangle$, and using
a lower bound on $\langle g_{i}, g_{i} \rangle$ from \cite{Rouse}. We obtain the following:
\begin{theorem}
\label{petbound}  
Suppose that $Q$ is a positive-definite form $Q$ in four-variables with prime discriminant $p$. Then
\[
  \langle C, C \rangle \leq \frac{1}{\min Q^{*}} + \frac{3216.66 M(25.09p^{35/6})}{p^{1/4}}.
\]
\end{theorem}
When $\min Q^{*}$ is small ($\ll p^{1/4}$), our upper bound on $\langle C, C \rangle$ matches the lower bound given by Waibel in \cite{Waibel}.
\begin{theorem*}[Waibel, 2020]
Suppose that $Q$ is a positive-definite form in $r$-variables. Then
\[
\langle C, C \rangle \gg \frac{1}{[{\rm SL}_{2}(\Z) : \Gamma_{0}(N(Q))]}
\left(\frac{N(Q)^{r/2}}{D(Q)} (\min Q^{*})^{1-\frac{r}{2}} + O(N(Q)^{\epsilon})\right).
\]
\end{theorem*}
The result above is stated using our normalization of the Petersson norm, which differs from that used by Waibel in \cite{Waibel}. In our case, $r = 4$, $N(Q) = D(Q) = p$.

As a consequence of Theorem~\ref{ineq} and Theorem~\ref{petbound}, we obtain a non-trivial bound on the sum of the integers not represented by $Q$.
\begin{corollary}
\label{sumofexcep}  
Assume the notation above. Then
\[
\sum_{\substack{n \\ r_{Q}(n) = 0}} n \ll \max \left\{ \frac{p^{3+\epsilon}}{(\min Q^{*})^{2}}, p^{5/2 + \epsilon} \right\}.
\]
\end{corollary}
The remainder of the paper is organized as follows: In Section 2, we introduce notation and review necessary background. In Section 3, we prove Theorem~\ref{petbound} and use that to prove Theorem~\ref{ineq}. In Section 4, we prove Corollary~\ref{sumofexcep}. Last, in Section 5 we provide an example of an infinite family of forms of prime discriminant along with the complete classification of their excepted values.

%In Section 2, we review necessary background. In Section 3, we prove Theorem~\ref{petbound} and use that to prove Theorem~\ref{ineq}. In Section 4, we prove Corollary~\ref{sumofexcep}.

\section{Background}

%\textcolor{blue}{There's some somewhat technical stuff from the intro we could
  %potentially move into this section. Should we?}
Throughout, we use the word \textit{``form''} to denote
a positive-definite, integer-valued, quaternary quadratic form $Q(\vec{x})
= \frac{1}{2} \vec{x}^{T} A \vec{x}$.  The \textbf{determinant} $D(Q)$ of the form is the determinant of $A$. The \textbf{level} $N=N(Q)$ of the form is the smallest integer such that $NA^{-1}$ is an integral matrix with even diagonal entries. Let $$r_Q(n) := \# \{ \vec{x} \in \mathbb Z^4 \vert Q(\vec{x})=n\}$$ denote the \textbf{representation number of $n$ by $Q$} and define the \textbf{theta series of $Q$}, $$\Theta_Q(z) = \displaystyle\sum_{n \geq 0}r_Q(n)q^n \hspace{1.0in} q = e^{2 \pi i z}.$$ It is well-known that $\Theta_{Q} \in \mathcal M_{2}(\Gamma_{0}(N(Q)), \chi_{D(Q)})$, and decomposes as
\[
\Theta_{Q}(z) = E(z) + C(z) = \sum_{n=0}^{\infty} a_{E}(n) q^{n} +
\sum_{n=1}^{\infty} a_{C}(n) q^{n}.
\]
where $E(z)$ is the Eisenstein series contribution and where $C(z)$ is the cusp form contribution.\\

Siegel expressed $a_{E}(n)$ as a product of local densities.
\begin{theorem}[Siegel]
\label{LocalDensities}  
We have
\[
  a_{E}(n) = \prod_{p \leq \infty} \beta_{p}(n)
\]
where the product is over primes $p$ and for finite $p$,
\[
  \beta_{q}(n) = \lim_{v \to \infty} \frac{\# \{ \vec{x} \in (\Z/p^{v} \Z)^{4} | Q(\vec{x}) \equiv n \pmod{p^{v}}\}}{p^{3v}},
\]
while
\[
\beta_{\infty}(n) = \frac{4 \pi^{2} n}{\sqrt{\det(A)}}.
\]
\end{theorem}
\begin{proof}
The formula for $\beta_{\infty}$ is a special case of Hilfssatz 72 of \cite{Siegel}. Note that Siegel's normalization of $A$ is different from ours, hence our correction with the factor of $4$. 
\end{proof}
In practice, for the finite places we rely on computational methods of Hanke, outlined in \cite{Hanke}. Let $R_{{p}^v}(n) := \left\{ \vec{x} \in (\mathbb Z/ {{p}^v} \mathbb Z)^4 \mid Q(\vec{x}) \equiv n \pmod{{p}^v} \right\}$ and set $r_{p^v}(n) := \# R_{{p}^v}(n)$. We say $\vec{x}\in R_{{p}^v}\left(n\right)$ 
\begin{itemize}
\item is of \textbf{Zero type} if $\vec{x} \equiv \vec{0} \pmod{ p}$, in which case we say $\vec{x} \in R_{{p}^v}^{\operatorname{Zero}}\left(n\right)$ with $r_{p^v}^{\operatorname{Zero}}\left(n\right) := \#R_{{p}^v}^{\operatorname{Zero}}\left(n\right)$; \\
\item is of \textbf{Good type} if ${p}^{v_j} x_j \not\equiv 0 \pmod{ p}$ for some $j \in\{1,2,3,4\}$, in which case we say $\vec{x} \in R_{{p}^v}^{\operatorname{Good}}\left(m\right)$ with $r_{p^v}^{\operatorname{Good}}\left(n\right) := \#R_{{p}^v}^{\operatorname{Good}}\left(n\right)$; \\
\item and is of \textbf{Bad type} otherwise, in which case we say $\vec{x} \in R_{{p}^v}^{\operatorname{Bad}}\left(n\right)$ with $r_{p^v}^{\operatorname{Bad}}\left(n\right) := \#R_{p^v}^{\operatorname{Bad}}(n)$.
\end{itemize}

If $r_{p^v}(n)>0$ for all primes $p$ and for all $v \in \mathbb{N}$, we say that $n$ is \textbf{locally represented}. If $Q(\vec{x})=0$ has only the trivial solution over $\mathbb{Z}_p$, we say that $p$ is an \textbf{anisotropic prime} for $Q$.

In the following theorems, we discuss reduction maps that allow for explicit calculation of local densities. Let the \textbf{multiplicity of a map $f : X \to Y$ at a given $y \in Y$} be $\#\{x \in X \mid f(x) = y\}$. If all $y \in Y$ have the same multiplicity $M$, we say that the map has multiplicity $M$.

\begin{theorem*}
We have $$r_{p^{k+ \ell}}^{\operatorname{Good}}(n) = p^{3 \ell} r_{p^k}^{\operatorname{Good}}(n)$$ for $k \geq 1$ for $p$ odd and for $k \geq 3$ for $p=2$. 
\end{theorem*}
\begin{proof}
See \cite[Lemma 3.2]{Hanke}.
\end{proof}
\begin{theorem*}
The map
\begin{eqnarray*}
\pi_Z : R_{p^k}^{\operatorname{Zero}}(n) & \to & R_{p^{k-2}} \left( \dfrac{n}{p^2} \right) \\
\vec{x} &\mapsto& p^{-1} \vec{x} \pmod{p^{k-2}}
\end{eqnarray*}
is a surjective map with multiplicity $p^4$.
\end{theorem*}
\begin{proof}
See \cite[pg. 359]{Hanke}.
\end{proof}

\begin{theorem*} \leavevmode
Write $Q$ in its local normalized form $$Q(\vec{x}) = \sum p^{v_j}Q_j(\vec{x_j})$$ where $\operatorname{dim}(Q_j) \leq 2$. Let $\mathbb S_0 = \{ j \vert v_j =0\}$, $\mathbb S_1 = \{ j \vert v_j=1\}$ and $\mathbb S_2 = \{j \vert v_j \geq 2\}$.
\begin{itemize} 
\item Bad-Type-I solutions occur when $\mathbb S_1 \neq \emptyset$ and $\vec{x}_{\mathbb S_1}\not\equiv \vec{0}$. The map 
\begin{eqnarray*}
\pi_{B'} : R_{{p}^k, Q}^{\operatorname{Bad-1}}(n) & \to & R_{{p}^{k-1}, Q'}^{\operatorname{Good}} \left(\dfrac{n}{p} \right)
\end{eqnarray*}
which is defined for each index $j$ by
\begin{center}
$$\begin{array}{ccc} {x_j} \mapsto p^{-1} {x_j} & v_j ' = v_j+1, & j \in \mathbb S_0 \\ {x_j} \mapsto {x_j} & v_j' = v_j-1, & j \not\in \mathbb S_0 \end{array}$$
\end{center}
is surjective with multiplicity $p^{s_1+s_2}$. \\

\item Bad-Type-II solutions can only occur when $\mathbb S_2 \neq \emptyset$ and involves either $\mathbb S_1 = \emptyset$ or $\vec{x}_{\mathbb S_1} \equiv \vec{0}$. The map
\begin{eqnarray*}
\pi_{B''} : R_{{p}^k, Q}^{\operatorname{Bad-II}}(n) & \to & R_{{p}^{k-2}, Q''}^{\vec{x}_{\mathbb S_2} \not\equiv \vec{0}} \left( \dfrac{n}{{{p}}^2} \right)
\end{eqnarray*}
which is defined for each index $j$ by
\begin{center}
$$\begin{array}{ccc} {x_j} \mapsto  p^{-1} {x_j} & v_j '' = v_j, & j \in \mathbb S_0 \cup \mathbb S_1 \\ {x_j} \mapsto {x_j} & v_j'' = v_j-2, & j \not\in \mathbb S_2 \end{array}$$
\end{center}
is surjective with multiplicity $p^{8-s_0-s_1}$.
\end{itemize}
\end{theorem*}
\begin{proof}
Again, see \cite[pg. 360]{Hanke}.
\end{proof}

To estimate $a_{C}(n)$, we use techniques from \cite{Rouse}, which we now
review. Assume that $D(Q)$ is a fundamental discriminant and decompose $C(z) = \sum_{i=1}^{s} c_{i} g_{i}(z)$
as a linear combination of newforms. Here $s = \dim S_{2}(\Gamma_{0}(N(Q)), \chi_{D_{Q}})$. The Deligne bound gives
\[
  |a_{C}(n)| \leq \left(\sum_{i=1}^s |c_{i}|\right) \tau(n) \sqrt{n}
\]
where $\tau(n)$ denotes the number of positive divisors of $n$. To estimate
$\sum_{i=1}^s |c_{i}|$ we find an upper bound on $\sum_{i=1}^s |c_{i}|^{2}$ and use the Cauchy-Schwarz inequality
\[
  \sum_{i=1}^s |c_{i}| \leq \sqrt{s} \sqrt{\sum_{i=1}^s |c_{i}|^{2}}.
\]
The Petersson inner product of two cusp forms $f, g \in S_{2}(\Gamma_{0}(N(Q)), \chi_{D_{Q}})$ is given by
\[
  \langle f, g \rangle = \frac{3}{\pi [\SL_{2}(\Z) : \Gamma_{0}(N(Q))]} \iint_{\mathbb{H} / \Gamma_{0}(N(Q))} f(x+iy) \overline{g(x+iy)} \, dx \, dy.
\]
Distinct newforms are orthogonal and this gives
\[
  \langle C, C \rangle = \sum_{i=1}^s |c_{i}|^{2} \langle g_{i}, g_{i} \rangle.
\]
\begin{proposition}
\label{Bbound}  
  If $p$ is a prime and $g_{i} \in S_{2}(\Gamma_{0}(p),\chi_{p})$ is a newform, then
\[
  \langle g_{i}, g_{i} \rangle \geq \frac{3p}{208 \pi^{4} (p+1) \log(p)}.
\]
\end{proposition}
\begin{proof}
  Theorem 8 of \cite{Rouse} gives that the value at $s = 1$ of the
  adjoint square $L$-function $L({\rm Ad}^{2} g_{i}, s)$ is $\frac{8 \pi^{4}}{3} (1+1/p) \langle g_{i}, g_{i} \rangle$. Proposition 11 of \cite{Rouse} shows that
  if $g_{i}$ is a non-CM newform, then $L({\rm Ad}^{2} g_{i}, 1) > \frac{1}{26 \log(p)}$. Finally, a CM newform in weight $k \geq 2$ must arise from a Hecke
  character attached to an imaginary quadratic field $K = \Q(\sqrt{d})$ and takes the value
\[
  \xi\left((a)\right) = \left(\frac{a}{|a|}\right)^{k-1}
\]
for all $a \in \mathcal{O}_{K}$ with $a \equiv 1 \pmod{\mathfrak{m}}$.
Here $\mathfrak{m}$ is the modulus of the Hecke character. Theorem 12.5 of \cite{Iwaniec} gives that the level of such a newform is $d(\mathcal{O}_{K}) |N(\mathfrak{m})|$ where $d(\mathcal{O}_{K})$
is the discriminant of $\mathcal{O}_{K}$. If the level is a prime $p$,
then $d(\mathcal{O}_{K}) = p$ and $|N(\mathfrak{m})| = 1$, but this
implies that $\xi((1)) = 1$ and $\xi((-1)) = -1$. Since the ideals
$(1)$ and $(-1)$ are equal this implies there are no prime
level CM newforms of weight $2$.
\end{proof}

\begin{remark}
  The above result is the key place in this paper where we use the assumption that the discriminant of $Q$ is prime. Most of the rest of the paper extends to any form with fundamental discriminant.
\end{remark}

It suffices to find an upper bound on $\langle C, C \rangle$. For this task,
we let $Q^{*} = \frac{1}{2} \vec{x}^{T} N(Q)A_Q^{-1} \vec{x}$ be the dual form to $Q$.
Let $\theta_{Q^{*}} = E^{*} + C^{*}$ be the decomposition into an Eisenstein
series and a cusp form.
The Fricke involution $W_{N}$ sends $\theta_{Q}$ to $\sqrt{N(Q)} \theta_{Q^{*}}$
and is an isometry for the Petersson inner product and therefore
$\langle C, C \rangle = N \langle C^{*}, C^{*} \rangle$. Moreover, the form
$C^{*} = \sum a_{C^{*}}(n) q^{n}$ has the property that $a_{C^{*}}(n) = 0$
if $\chi_{D(Q)}(n) = 0$ or $1$ and so $C^{*} \in S_{2}^{-}(\Gamma_{0}(N(Q)), \chi_{D(Q)})$ (see Proposition 15 of \cite{Rouse}). Proposition 14 of \cite{Rouse} then gives that
\[
\langle C^{*}, C^{*} \rangle
= \frac{1}{[\SL_{2}(\Z) : \Gamma_{0}(N(Q))]}
\sum_{n=1}^{\infty} \frac{2^{\omega(\gcd(n,N(Q)))} a_{C^{*}}(n)^{2}}{n} \sum_{d=1}^{\infty}
\psi\left(d \sqrt{\frac{n}{N(Q)}}\right),
\]
where
\[
  \psi(x) = -\frac{6}{\pi} x K_{1}(4 \pi x) + 24x^{2} K_{0}(4 \pi x),
\]
where $K_{0}$ and $K_{1}$ are the usual $K$-Bessel functions. We use these formulas to estimate $\langle C, C \rangle$.

%%%%%%%%
%%%%%%%%
\section{The Largest Excepted Value}
In this section we explicitly find $a_E(m)$ and give bounds on $a_C(m)$ to prove Theorems \ref{ineq} and \ref{petbound}.
%Throughout, let $Q= \frac{1}{2} \vec{x}^T A \vec{x}$ be a form. The theta series associated to $Q$, $\Theta_Q(z) = \sum_{n=0}^\infty r_Q(n) q^n \in \mathcal M_2(\Gamma_0(p), \chi_p)$. We decompose the theta series in the natural fashion: $$\Theta_Q(z) = E(z) + C(z)$$ where $E(z) = \sum_{n=0}^\infty a_E(n) q^n$ is the Eisenstein series contribution and $C(z) = \sum_{n=1}^\infty a_C(n) q^n$ is the cusp form contribution.
%%%%%%%%
\subsection{The Eisenstein Series}
First, the Eisenstein subspace is $2$ dimensional (see the details of the Proof of Theorem 5 of \cite{BB}) and is spanned by
\begin{eqnarray*}
G(z) & = & 1+ \dfrac{2}{L(-1, \chi_p)} \displaystyle\sum_{n=1}^\infty \left( \displaystyle\sum_{d \vert n} d \chi_p(d)\right)q^n, \textrm{ and}\\
H(z) & = & \displaystyle\sum_{n=1}^\infty \left( \displaystyle\sum_{ d \vert n} d\chi_p(n/d)\right)q^n.
\end{eqnarray*}
Therefore, to express $E(z)$ as a linear combination of $G(z)$ and $H(z)$, it suffices to note that $a_E(0)=1$ due to $Q$ being positive definite, and to find $a_E(1)$:
\begin{theorem}
\label{E1} $a_E(1) = \dfrac{-2(p-1)}{L(-1, \chi_p)}$.
\end{theorem}
\begin{proof}
Siegel's local density formula (Theorem \ref{LocalDensities}) gives
\begin{eqnarray*}
a_E(1) & = & \displaystyle\prod_{q \leq \infty} \beta_q(1) \\
& = & \beta_{\infty}(1) \beta_2(1) \beta_p(1) \left( \displaystyle\prod_{q \nmid 2p}\beta_q(1) \right).
\end{eqnarray*}
We have $\beta_{\infty}(1) = \frac{4 \pi^2}{\sqrt{p}}$. For $\beta_2(1)$, noting via the language of \cite{Hanke} that all solutions are of Good-type, one can see that $\beta_2(1) = \frac{4-\chi_p(2)}{4}$. The local density at $p$ can be computed using \cite{Yang} which gives $\beta_p(1) = 1-\frac{1}{p}$ . Last, for the remaining infinitely many primes, taking a restricted look at Theorem \ref{LocalDensities} gives 
\begin{eqnarray*}
\displaystyle\prod_{q \nmid 2p} \beta_q(1) &= & \displaystyle\prod_{q \nmid 2p} \left( 1- \frac{\chi_p(q)}{q^2}\right)\\
& = & L(2, \chi_p)^{-1} \left( \dfrac{4}{4-\chi_p(2)}\right).
\end{eqnarray*}Techniques outlined on page 104 of \cite{Iwasawa} show that $$L(2, \chi_p) = -\dfrac{2\pi^2}{p\sqrt{p}}L(-1, \chi_p)$$ and hence $$\displaystyle\prod_{q \nmid 2p} \beta_q(1) = -\dfrac{p\sqrt{p}}{2 \pi^2 L(-1, \chi_p)} \left( \frac{4}{4-\chi_p(2)} \right).$$
Putting everything together, we have
\begin{eqnarray*}
a_E(1) & = & \beta_{\infty}(1) \beta_2(1) \beta_p(1) \left( \displaystyle\prod_{q \nmid 2p}\beta_q(1) \right)\\
& = & \left( \frac{4 \pi^2}{\sqrt{p}}\right) \left( \dfrac{4-\chi_p(2)}{4}\right) \left( 1-\frac{1}{p}\right) \left(-\dfrac{p\sqrt{p}}{2 \pi^2 L(-1, \chi_p)} \left( \frac{4}{4-\chi_p(2)} \right)\right)\\
& = & \frac{-2(p-1)}{L(-1, \chi_p)},
\end{eqnarray*}
as originally claimed.
\end{proof}
As a result of Theorem \ref{E1}, we now can write $$E(z)= G(z) - \frac{2p}{L(-1, \chi_p)} H(z).$$ Moreover, the $n$th coefficient of the Eisenstein series is $$a_E(n) = -\dfrac{2}{L(-1, \chi_p)} \displaystyle\sum_{d \vert n} d(p \chi_p(n/d)-\chi_p(d)).$$ Replacing $d$ with $n/d$ simplifies this to $$a_E(n) = -\dfrac{2}{L(-1, \chi_p)} \displaystyle\sum_{d \vert n} \left(\frac{pn}{d}-d \right)\chi_p(d).$$
\begin{theorem}
\label{eisineq}  
For all $n$, $$a_E(n) \geq \dfrac{24}{p^{3/2}}(p-1) \phi(n).$$ 
\end{theorem}
\begin{proof}
Throughout we write $n^\ast = n/p^{\ord_p(n)}$. We first note %show that for all $n$, $$a_E(n) \geq \dfrac{-2}{L(-1, \chi_p)}(p-1)\phi(n).$$
\begin{eqnarray*}
a_E(n) %&=& -\dfrac{2}{L(-1, \chi_p)} \displaystyle\sum_{d \vert n} \left(\frac{pn}{d}-d \right)\chi_p(d)\\
&=& -\dfrac{2}{L(-1, \chi_p)} \displaystyle\sum_{d \vert n^\ast} \left(\frac{pn}{d}-d \right)\chi_p(d)\\
%& = & \dfrac{-2}{L(-1, \chi_p)} pn \displaystyle\sum_{d \vert n^\ast} \dfrac{\chi_p(d)}{d} - \dfrac{-2}{L(-1,\chi_p)}\displaystyle\sum_{d \vert n^\ast}\dfrac{n^\ast}{d}\chi_p \left( \dfrac{n^\ast}{d} \right)\\
& = & \dfrac{-2}{L(-1,\chi_p)}(pn-n^\ast\chi_p(n^\ast)) \displaystyle\sum_{d \vert n^\ast} \dfrac{\chi_p(d)}{d}\\
& \geq & \dfrac{-2}{L(-1, \chi_p)}(pn-n^\ast \chi_p(n^\ast)) \displaystyle\sum_{d \vert n^\ast} \dfrac{\mu(d)}{d}\\
& \geq & \dfrac{-2}{L(-1, \chi_p)}(pn-n^\ast) \dfrac{\phi(n^\ast)}{n^\ast}.
\end{eqnarray*}
When $\ord_p(n)=0$, $$\dfrac{-2}{L(-1, \chi_p)}(pn-n^\ast) \dfrac{\phi(n^\ast)}{n^\ast} = \dfrac{-2}{L(-1, \chi_p)}(p-1) \phi(n)$$ and if $\ord_p(n)=k \geq 1$ then
\begin{eqnarray*}
\dfrac{-2}{L(-1, \chi_p)}(pn-n^\ast) \dfrac{\phi(n^\ast)}{n^\ast} & \geq & \dfrac{-2}{L(-1, \chi_p)}(p^{k+1}n^\ast - n^\ast) \dfrac{\phi (n^\ast)}{n^\ast} = \dfrac{-2}{L(-1, \chi_p)}(p^{k+1}-1)\phi(n^\ast)\\
& \geq & \dfrac{-2}{L(-1, \chi_p)} \dfrac{p^{k+1}-1}{p^k-p^{k-1}} \phi(n)\\
%& \geq & \dfrac{-2}{L(-1, \chi_p)}p \phi(n)\\
& \geq & \dfrac{-2}{L(-1, \chi_p)}(p-1) \phi(n).
\end{eqnarray*}
As $L(-1, \chi_p) = -\frac{p^{3/2}}{2\pi^2}L(2, \chi_p)$ and $\vert L(2, \chi_p) \vert \leq \sum_{n=1}^\infty \frac{1}{n^2} = \frac{\pi^2}{6}$, $$-\dfrac{2}{L(-1, \chi_p)} \geq \frac{4\pi^2}{p^{3/2}} \cdot \frac{6}{\pi^2}= \frac{24}{p^{3/2}}$$
and the claim holds.
\end{proof}
%%%%%%%
\subsection{The Cusp Series}
We begin with $$C(z)=\displaystyle\sum_{i=1}^s c_ig_i$$ where $c_i \in \mathbb C$, the $g_i(z)$ are newforms in $S_2(\Gamma_0(p), \chi_p)$ and $s = \textrm{dim } S_2(\Gamma_0(p), \chi_p) = 2 \left\lfloor \tfrac{p-5}{24} \right\rfloor$. Bounds of Deligne give $$\vert a_C(n) \vert \leq C_Q \tau(n) \sqrt{n},$$ where $C_Q=\displaystyle\sum_{i=1}^s \vert c_i \vert$. In order to determine an upper bound on $C_Q$ we first note that the Petersson norm of $C$ is $$\langle C,C \rangle = \displaystyle\sum_{i=1}^s \vert c_i \vert^2 \langle g_i,g_i \rangle.$$ 
The Cauchy-Schwarz inequality gives $$C_Q \leq \sqrt{\dfrac{As}{B}}$$ where $\langle C,C \rangle \leq A$ and $\langle g_i, g_i \rangle \geq B$ for all $i$. We will take $s \leq \tfrac{p}{12}$. Proposition~\ref{Bbound} above gives
$B = \frac{3}{208 \pi^{4}} \cdot \frac{p}{(p+1) \log(p)}$.
%\textcolor{blue}{We also want a theorem giving a bound on $\langle C, C \rangle$ since that will be used for the sum of the exceptions.} \textcolor{red}{Agreed. But are you saying we alter the intro and what this section is going to do, or what?}
 More taxing is effectively computing $A$, which we discuss in the following subsection. 
%We provide details in a separate section, but will conclude.
%\begin{theorem}
%For $p \geq 101$, $$\langle C,C \rangle \leq \frac{1}{a_{4}^{*}} +
%  \frac{M(25.09 p^{35/6}) (a_{4}^{*})^{1/3}}{\sqrt{p}}
%  \left(918.87 + 353.53 (a_{4}^{*})^{-1/12} \log(p) + 125.74 p^{1/4} (a_{4}^{*})^{-1/3} + 949.86 (a_{4}^{*})^{-1/3}\right)$$  where $M(n) = \max_{1 \leq m\leq n} \tau(m)$.
%\end{theorem}
\subsection{Computing $A \geq \langle C, C \rangle$}
Recall that $$\langle C,C \rangle = \dfrac{p}{p+1} \displaystyle\sum_{n=1}^\infty \dfrac{2^{\omega(\gcd(n,p))}a_{C^\ast}(n)^2}{n} \displaystyle\sum_{d=1}^\infty \psi \left( d\sqrt{\frac{n}{p}} \right)$$ where $\omega$ counts the number of distinct prime divisors and where $$\psi(x) = -\frac{6}{\pi}xK_1(4 \pi x)+24x^2 K_0(4 \pi x).$$
Note that $\psi(x) \leq 24x^{2} K_{0}(4 \pi x)$ and (by equation (9) on page 1724 of \cite{Rouse}), $K_{0}(x) \leq \sqrt{\frac{\pi}{2x}} e^{-x}$. Therefore $\psi(x) \leq 6 \sqrt{2} x^{3/2} e^{-4 \pi x}$.
We begin by bounding the second sum in $d$:
\begin{lemma}
  The function $\displaystyle\sum_{d=1}^\infty \psi(dx)$ is decreasing for
  $x > 0$. Also, for all $x$, $\displaystyle\sum_{d=1}^\infty \psi(dx) \leq \frac{3}{4 \pi^2}$ and if $x>0.5$ $\displaystyle\sum_{d=1}^\infty \psi(dx) \leq 9x^{3/2}e^{-4\pi x}$. 
\end{lemma}
\begin{proof}
  Choosing $\psi(0) = \frac{3}{2 \pi^{2}}$ and $\psi(-x) = \psi(x)$ if $x < 0$
  makes $\psi$ into a continuous, even function. Let $f(x) = \sum_{d=1}^{\infty} \psi(dx)$. The Weierstrass $M$-test shows that $\sum_{n=1}^{\infty} n \psi'(nx)$
converges uniformly on compact subsets of $(0,\infty)$ and thus $f'(x)
= \sum_{n=1}^{\infty} n \psi'(nx)$. We have that
\[
  \psi'(x) = 72 x K_{1}(4 \pi x) \left(\frac{K_{0}(4 \pi x)}{K_{1}(4 \pi x)} - \frac{4}{3} \pi x\right).
\]
So $K_{1}(4 \pi x) > 0$ for all $x > 0$ and \cite{Soni}
shows that $\tfrac{K_{1}(x)}{K_{0}(x)} > 1$ for $x > 0$. This implies that
$f'(x) < 0$ if $x > \frac{3}{4 \pi}$.

The Poisson summation formula applied to $x \psi'(x)$ gives
\[
  \sum_{n=1}^{\infty} n \psi'(nx) = \sum_{n=1}^{\infty} \frac{-18n^{2} x (n^{2}-6x^{2})}{\pi^{2} (n^{2} + 4x^{2})^{7/2}}.
\]
If $x < 1/\sqrt{6}$, every term on the right hand side is negative and therefore $f'(x) < 0$. Since $1/\sqrt{6} > \frac{3}{4 \pi}$ we have $f'(x) < 0$ for all
$x$ and this proves that $\sum_{d=1}^{\infty} \psi(dx)$ is decreasing.

Similarly, applying Poisson summation to $\displaystyle\sum_{d=1}^\infty \psi(dx)$ we see $$\displaystyle\sum_{d=1}^\infty \psi(dx) = \dfrac{3}{4 \pi^2} + \displaystyle\sum_{d=1}^\infty \frac{1}{x}\hat{\psi}(d/x),$$ where $\hat{\psi}$ is the Fourier transform of $\psi$ and $$\hat{\psi}(y) = -\dfrac{9y^2}{\pi^2(4+y^2)^{5/2}}.$$ Thus, we have
\begin{eqnarray*}
\displaystyle\sum_{d=1}^\infty \psi(dx) & = & \dfrac{3}{4\pi^2} - \dfrac{9}{\pi^2}\displaystyle\sum_{d=1}^\infty \dfrac{(d/x)^2}{(4+(d/x)^2)^{5/2}} \cdot \frac{1}{x}\\
& = & \dfrac{3}{4\pi^2}- \dfrac{9}{\pi^2}\displaystyle\sum_{d=1}^\infty \dfrac{d^2x^2}{(4x^2+d^2)^{5/2}}\\
& \leq & \dfrac{3}{4\pi^2}.
\end{eqnarray*}
At the same time
\begin{eqnarray*}
\displaystyle\sum_{d=1}^\infty \psi(dx) & \leq & 6 \sqrt{2} \displaystyle\sum_{d=1}^\infty (dx)^{3/2}e^{-4 \pi dx} \\
& \leq & 6\sqrt{2} x^{3/2} \displaystyle\sum_{d=1}^\infty d^2e^{-4\pi dx}\\
& \leq & \dfrac{6\sqrt{2}x^{3/2}(1+e^{-4\pi x})e^{-4\pi x}}{(1-e^{-4\pi x})^3}.
\end{eqnarray*}
Hence, if $x \geq 0.5$ we have $$\displaystyle\sum_{d=1}^\infty \psi(dx) \leq 9x^{3/2}e^{-4\pi x}$$ as desired.
\end{proof}

For the sum in $n$, we use $\vert a_{C^\ast}(n) \vert \leq r_{Q^{\ast}}(n)+a_{E^{\ast}}(n)$ which in turn will give the bound $$\vert a_{C^\ast}(n) \vert^2 \leq r_{Q^{\ast}}(n)^2+2r_{Q^\ast}(n)a_{E^\ast}(n) + a_{E^\ast}(n)^2 \leq 2r_{Q^{\ast}}(n)^2+2a_{E^{\ast}}(n)^2.$$ We now bound the representation number and Eisenstein components separately. We begin with the Eisenstein contribution to $\langle C,C \rangle$:
\begin{lemma}
\label{eispart}  
  $$\dfrac{p}{p+1} \displaystyle\sum_{n=1}^\infty \dfrac{2^{\omega(\gcd(n,p))} \cdot 2a_{E^\ast}(n)^2}{n} \displaystyle\sum_{d=1}^\infty \psi \left( d\sqrt{\dfrac{n}{p}} \right) \leq \dfrac{337.26 \log(p+2) + 206.64}{p}.$$
\end{lemma}
\begin{proof}
  The proof of Proposition 15 of \cite{Rouse} shows that $E^{\ast}(z) = \sum_{n=0}^{\infty} a_{E^{*}}(n) q^{n}$ has the property that $a_{E^{*}}(n) = 0$ if
  $n$ is a quadratic residue modulo $p$. Equation (14) of \cite{BB} gives a formula for the unique
  Eisenstein series with this property and implies that
\[
a_{E^\ast}(n) = \frac{2}{L(-1,\chi_{p})} \sum_{d | n} d (\chi_{p}(d) - \chi_{p}(n/d))
\]
and hence   
\begin{eqnarray*}
a_{E^\ast}(n) & \leq &- \dfrac{4}{L(-1, \chi_p)}\sigma(n) \\
& \leq & \dfrac{4\pi^4}{3p^{3/2}} \sigma(n).
\end{eqnarray*}
Robin's inequality \cite{Robin} states that for $n \geq 3$,
$\sigma(n) \leq e^{\gamma} n \log(\log(n)) + \frac{0.6483 n}{\log(\log(n))}$.
This implies that $\frac{\sigma(n)}{n \sqrt{\log(n+2)}}$ is decreasing
if $n > 1652$ and in particular, $\sigma(n) \leq 1.44n \sqrt{\log(n+2)}$ for
$n \geq 1$. Thus 
\begin{eqnarray*}
  \dfrac{p}{p+1} \displaystyle\sum_{n=1}^\infty \dfrac{2^{\omega(\gcd(n,p))} \cdot 2a_{E^\ast}(n)^2}{n} \displaystyle\sum_{d=1}^\infty \psi \left( d\sqrt{\dfrac{n}{p}} \right) & \leq & \dfrac{32 \pi^8}{9p^3}\displaystyle\sum_{n=1}^{\infty} \dfrac{2^{\omega(\gcd(n,p))} \sigma(n)^2}{n} \sum_{d=1}^{\infty} \psi\left(d \sqrt{\frac{n}{p}}\right).\\
\end{eqnarray*}
For the terms with $n < p$ we use that $\sum_{d=1}^{\infty} \psi(dx)$ is decreasing. Then the contribution for
$ap \leq n \leq bp$ is bounded by
\[
(1.44)^{2} \cdot \frac{32 \pi^{8}}{9p^{3}} \sum_{n=ap}^{bp}
n \log(n+2) \sum_{d=1}^{\infty} \psi\left(d \sqrt{\frac{n}{p}}\right)
\leq \frac{69958}{p^{3}} \cdot \left[bp \log(bp+2)\right] (b-a)p \sum_{d=1}^{\infty} \psi(d \sqrt{a}).
\]
We apply this for $(a,b) = (0,0.01), (0.01,0.1), (0.1,0.2), \ldots,
(0.9,1)$ and obtain a bound of $$\frac{69958 \log(p+2)}{p} \cdot 0.00150889
\leq \frac{105.56 \log(p+2)}{p}.$$

Now, for $n \geq p$, we have $2^{\omega(\gcd(n,p))} \leq 2$. We also use the
bound $\sum_{d=1}^{\infty} \psi(d \sqrt{n/p}) \leq 9 (n/p)^{3/4} e^{-4 \pi \sqrt{n/p}}$. This gives that the contribution for $n \geq p$ is bounded by
\begin{align*}
  & \frac{(1.44)^{2} 64 \pi^{8}}{p^{3}} \sum_{k=1}^{\infty}
  \sum_{n=k^{2} p}^{(k+1)^{2} p} n \log(n+2) \cdot \left(\frac{n}{p}\right)^{3/4} e^{-4 \pi k}\\
  &\leq \frac{1259227}{p^{15/4}} \sum_{k=1}^{\infty}
  (p (k+1)^{2})^{7/4} \log((k+1)^{2} (p + 2)) (2k+1)p e^{-4 \pi k}\\
  &\leq \frac{1259227}{p} \sum_{k=1}^{\infty} (k+1)^{7/2} (2k+1) \log((k+1)^{2} (p + 2)) e^{-4 \pi k}\\
  &\leq \frac{1259227}{p} \left[\sum_{k=1}^{\infty}
    2 \log(k+1) (k+1)^{7/2} (2k+1) e^{-4 \pi k} + \log(p+2) \sum_{k=1}^{\infty}
    (k+1)^{7/2} (2k+1) e^{-4 \pi k}\right]\\
  &\leq \frac{1259227}{p} \left[0.0001641 + 0.0001184 \log(p+2)\right]
  \leq \frac{206.64 + 231.7 \log(p+2)}{p}.
\end{align*}
Combining the contribution for $n < p$ and $n \geq p$ gives the desired result.
\end{proof}

To estimate the other contribution to $\langle C, C \rangle$, namely
\[
  \frac{p}{p+1} \sum_{n=1}^{\infty} \frac{2^{\omega(\gcd(n,p))} \cdot 2r_{Q^{\ast}}(n)^{2}}{n} \sum_{d=1}^{\infty} \psi\left(d \sqrt{\frac{n}{p}}\right)
\]
we will need to bound $\sum_{n \leq x} r_{Q^{\ast}}(n)^{2}$. We do this via the
simple inequality
\[
\sum_{n \leq x} r_{Q^{\ast}}(n)^{2} \leq
\left(\max_{n \leq x} r_{Q^{\ast}}(n)\right) \left(\sum_{n \leq x} r_{Q^{\ast}}(n)\right).
\]
We will proceed to estimate the two factors on the right hand side. To do so
we write $Q$ as %assume that $Q$ is Korkin-Zolotarev reduced \textcolor{red}{Do we absolutely need that terminology? Because then thinking more background needed?}. In particular, we write
\[
Q = a_{1} (x_{1} + m_{12} x_{2} + m_{13} x_{3} + m_{14} x_{4})^{2}
+ a_{2} (x_{2} + m_{23} x_{3} + m_{24} x_{4})^{2} + a_{3} (x_{3} + m_{34} x_{4})^{2}
+ a_{4} x_{4}^{2}.
\]
Here $a_{1}$ is the minimum nonzero value of $Q$, $a_{i+1}/a_{i} \geq 3/4$,
and $a_{1} a_{2} a_{3} a_{4} = \det(A/2) = p/16$. Also, $|m_{ij}| \leq 1/2$ for all $i$ and $j$. (That we can write $Q$ in such a form follows from Theorem 2.1.1 of \cite{Kitaoka}.)

This corresponds to
writing $Q = \frac{1}{2} \vec{x}^{T} A \vec{x}$,
where $\frac{1}{2} A = M^{T} DM$, where $D$ is diagonal with entries
$a_{1}, \ldots, a_{4}$. This gives $\frac{1}{2} pA^{-1} = \frac{p}{4} M^{-1} D^{-1} (M^{T})^{-1}$ and yields the formula
\[
Q^{*} = a_{1}^{*} x_{1}^{2} + a_{2}^{*} (x_{2} + n_{21} x_{1})^{2}
+ a_{3}^{*} (x_{3} + n_{31} x_{1} + n_{32} x_{2})^{2} + a_{4}^{*} (x_{4} + n_{41} x_{1} + n_{42} x_{2} + n_{43} x_{3})^{2},
\]
where $a_{i}^{*} = \frac{p}{4a_{i}}$. %This representation gives the bounds
The following Lemma is an explicit quantitative version of the first part of
GH from MO's answer \href{https://mathoverflow.net/questions/256576/which-quaternary-quadratic-form-represents-n-the-greatest-number-of-times}{here} to a question on MathOverflow \cite{over}. %It provides an upper bound on $r_{Q^{\ast}}(n)$ that is
%in some cases superior to that of the previous lemma.

\begin{lemma}
\label{rQbound2}    
Define $M(n) = \max_{1 \leq m \leq n} \tau(m)$. Then for $p \geq 17$
\[
  r_{Q^\ast}(n) \leq 2 \left(2\sqrt{\frac{n}{a_{1}^{*}}} + 1\right) \left(2 \sqrt{\frac{n}{a_{2}^{*}}} + 1\right) M(25.09np^{29/6}).
\]
\end{lemma}

Note that $M(n) = O(n^{\epsilon})$. More precisely,
$M(n) \leq n^{\frac{1.538 \log(2)}{\log \log(n)}}$ for $n \geq 3$. (See \cite{NR}.)

\begin{proof}
From the formula
\[
Q^{*} = a_{1}^{*} x_{1}^{2} + a_{2}^{*} (x_{2} + n_{21} x_{1})^{2}
+ a_{3}^{*} (x_{3} + n_{31} x_{1} + n_{32} x_{2})^{2} + a_{4}^{*} (x_{4} + n_{41} x_{1} + n_{42} x_{2} + n_{43} x_{3})^{2},
\]
we know that the number of choices for $(x_{1},x_{2})$ so that $Q^{*}(x_{1},x_{2},x_{3},x_{4}) = n$ is at most
\[
  \left(2\sqrt{\frac{n}{a_{1}^{*}}} + 1\right) \left(2 \sqrt{\frac{n}{a_{2}^{*}}} + 1\right).
\]
For each such pair $(x_{1},x_{2})$, we specialize $Q^{*}$ at this value leaving
a binary quadratic polynomial in the two variables $x_{3}$, $x_{4}$. More precisely, we wish to count integer solutions to
\[
  P(x_{3},y_{3}) := ax_{3}^{2} + bx_{3} x_{4} + cx_{4}^{2} + dx_{3} + ex_{4} + f = 0,
\]
where
\begin{align*}
& a = a_{3}^{*} + a_{4}^{*} n_{43}^{2} \quad
  b = 2 a_{4}^{*} n_{43} \quad c = a_{4}^{*} \quad d = 2 a_{3}^{*} n_{31} x_{1} + 2 a_{3}^{*} n_{32} x_{2} + 2a_{4}^{*} n_{41} n_{43} x_{1} + 2a_{4}^{*} n_{42} n_{43} x_{2}\\
  & e = 2a_{4}^{*} n_{43} (n_{41} x_{1} + n_{42} x_{2}) \quad f = Q^{*}(x_{1},x_{2},0,0) - n.
\end{align*}
Lemma 8 of \cite{BP} shows that if $\Delta = b^{2} - 4ac$,
$\xi = (be-2cd)/\Delta$ and $\eta = (bd-2ae)/\Delta$, then
\[
  P(x,y) = \frac{(2a(x+\xi) + b(y+\eta))^{2} - \Delta (y+\eta)^{2}}{4a} + P(-\xi,-\eta).
\]
Note that $\Delta = b^{2} - 4ac < 0$ since $Q^{*}$ is positive-definite.  
It follows that the number of solutions to $P(x,y) = 0$ is at most the number
of solutions to $X^{2} - \Delta Y^{2} = -4a \Delta^{2} P(-\xi,-\eta)$. The number
of such solutions is at most the number of elements in the ring of integers
of $K = \Q(\sqrt{-\Delta})$ with norm $m := -4a \Delta^{2} P(-\xi,-\eta)$. If
$K \ne \Q(\sqrt{-3})$ or $\Q(\sqrt{-1})$, the number of elements of norm
$m$ is at most $2 \tau(m)$ because there are at most $\tau(m)$ ideals with norm $m$ and $\mathcal{O}_{K}$ has only two units.

If $K = \Q(\sqrt{-1})$ then $-\Delta$ is an even square and it suffices to show
that the number of solutions to $X^{2} + 4Y^{2} = m$ is at most $2 \tau(m)$. The fact that $X^{2} + 4Y^{2}$ is anisotropic at $2$ reduces this to proving the
result for $m$ odd. The number of elements with norm $m$ in $\Z[i]$ is
at most $4 \tau(m)$, but $\Z[i]$ contains two sublattices with index $2$
(namely, $\langle 2, i \rangle$, $\langle 1, 2i \rangle$) corresponding to the quadratic form $X^{2} + 4Y^{2}$. These two sublattices intersect only in elements with
even norm, and this shows that if $m$ is odd, $r_{X^{2}+4Y^{2}}(m) \leq \frac{1}{2} r_{X^{2}+Y^{2}}(m) \leq \frac{1}{2} \cdot (4 \tau(m))$ giving the desired result. A similar argument applies if $K = \Q(\sqrt{-3})$.

It suffices to bound
\[
  -4a \Delta^{2} P(-\xi,-\eta) = 4a(4ac-b^{2})((b^{2}-4ac)f + ae^{2} - bde + cd^{2}).
\]
We have $a_{3}^{*} = \frac{p}{4a_{3}} \leq \frac{p}{4(9/16)} = \frac{4p}{9}$ and $\frac{p}{16} = a_{1} a_{2} a_{3} a_{4} \leq \left(\frac{4}{3}\right)^{6} a_{4}^{4}$ which give $a_{4} \geq \frac{3 \sqrt{3} p^{1/4}}{16}$ and
so $a_{4}^{*} \leq \frac{4 p^{3/4}}{3 \sqrt{3}}$. Therefore
\begin{align*}
a &= a_{3}^{*} + a_{4}^{*} n_{43}^{2}
\leq \frac{4p}{9} + \frac{p^{3/4}}{3 \sqrt{3}} \leq 0.54p \qquad
|b| = 2a_{4}^{*} |n_{43}| \leq \frac{4 p^{3/4}}{3 \sqrt{3}} \qquad
c = a_{4}^{*} \leq \frac{4 p^{3/4}}{3 \sqrt{3}}\\
|b^{2} - 4ac| &= |-4a_{3}^{*} a_{4}^{*}| \leq \frac{64 p^{7/4}}{27 \sqrt{3}}\\
\end{align*}
To bound $d$ and $e$ we use that $\frac{1}{a_{1}^{*}} = \frac{4a_{1}}{p}$
and $\frac{1}{a_{2}^{*}} = \frac{4a_{2}}{p}$ together with the bounds $a_{1} \leq \frac{4}{3 \sqrt{3}} p^{1/4}$, $a_{2} \leq \frac{2^{2/3}}{3} p^{1/3}$.
This yields
\begin{align*}
  |d| &= 2a_{3}^{*} |n_{31} x_{1}| + 2a_{3}^{*} |n_{32} x_{2}| +
  2a_{4}^{*} |n_{41} n_{43} x_{1}| + 2a_{4}^{*} |n_{42} n_{43} x_{2}|
  \leq \sqrt{\frac{n}{a_{1}^{*}}} \left(\frac{2p}{3} + \frac{7 p^{3/4}}{6 \sqrt{3}}\right) + \sqrt{\frac{n}{a_{2}^{*}}} \left(\frac{2p}{3} + \frac{p^{3/4}}{\sqrt{3}}\right)\\
  &\leq \sqrt{n} \cdot \frac{4 p^{-3/8}}{3^{3/4}} \left(\frac{2p}{3} + \frac{7 p^{3/4}}{6 \sqrt{3}}\right) + \sqrt{n} \cdot \frac{2^{4/3} p^{-1/3}}{\sqrt{3}} \cdot \left(\frac{2p}{3} + \frac{p^{3/4}}{\sqrt{3}}\right) \leq 2.941 n^{1/2} p^{2/3}\\ 
|e| &= 2a_{4}^{*} |n_{43} (n_{41} x_{1} + n_{42} x_{2})|  \left(\frac{2p}{3} + \frac{p^{3/4}}{\sqrt{3}}\right)\\
& \leq 2 \cdot \frac{4p^{3/4}}{3 \sqrt{3}} \cdot (1/2) \left(\frac{7}{8} \sqrt{\frac{n}{a_{1}^{*}}} + \frac{3}{4} \sqrt{\frac{n}{a_{2}^{*}}}\right) \leq 0.946 n^{1/2} p^{5/12}.
\end{align*}
Finally, $-n \leq f \leq 0$. Putting everything together we have:
\begin{align*}
  & 4a(4ac-b^{2})((b^{2}-4ac)f + ae^{2} - bde + cd^{2})\\
  &\leq 4(0.54p) \cdot \frac{64 p^{7/4}}{27 \sqrt{3}}
  \left(\frac{64p^{7/4}}{27 \sqrt{3}} n + 0.484n p^{11/6}
  + 2.142np^{11/6} + 6.659np^{25/12}\right)\\
  &\leq 25.09np^{29/6}.
\end{align*}
%\textcolor{blue}{We could make this $15.65np^{5}$, which isn't really that much bigger.}
It follows that for each $x_{1}$, $x_{2}$, the number of solutions
to $Q^{\ast}(x_{1},x_{2},x_{3},x_{4}) = n$ is $\leq 2 M(25.09np^{29/6})$. %The desired
%result follows.
\end{proof}

\begin{lemma}
\label{sumbound}  
For $p \geq 17$, we have $$\displaystyle\sum_{n \leq x}r_{Q^\ast}(n) \leq 64x^2p^{-3/2}+75.52x^{3/2}p^{-1}+30.49xp^{-1/2}+8.11x^{1/2}+1.$$
\end{lemma}
\begin{proof}
We begin with expanding a rather simple bound and perhaps inelegantly moving term by term:
\begin{eqnarray*}
\displaystyle\sum_{n \leq x}r_{Q^\ast}(n) & \leq & \displaystyle\prod_{i=1}^4 \left( 4\sqrt{\frac{xa_i}{p}}+1 \right)\\
& \leq & 256 \sqrt{a_1a_2a_3a_4}x^2p^{-2} + 64(\sqrt{a_1a_2a_3}+\sqrt{a_1a_2a_4}+\sqrt{a_1a_3a_4}+\sqrt{a_2a_3a_4})x^{3/2}p^{-3/2}\\
&&+ 16(\sqrt{a_1a_2}+\sqrt{a_1a_3}+\sqrt{a_1a_4}+\sqrt{a_2a_3}+\sqrt{a_2a_4}+\sqrt{a_3a_4})x p^{-1}\\
&&+ 4(\sqrt{a_1} + \sqrt{a_2} + \sqrt{a_3} + \sqrt{a_4}) \sqrt{x} p^{-1/2} +1.
\end{eqnarray*}
Using $a_1a_2a_3a_4= \tfrac{p}{16}$ immediately bounds the first term by $64x^2p^{-3/2}$, which in turn leads to  $\sqrt{a_1a_2a_3} \leq \frac{1}{3^{3/4}} p^{3/8}$. As the minimum of $a_3$ is $\tfrac{9}{16}$ we have $\sqrt{a_1a_2a_4} \leq \tfrac{1}{3}p^{1/2}$. Similarly, with a minimum of $a_2$ of $\tfrac{3}{4}$ we have $\sqrt{a_1a_3a_4} \leq \tfrac{1}{2\sqrt{3}}p^{1/2}$. Last, considering the minimum of $a_1$ is $1$, we see $\sqrt{a_2a_3a_4} \leq \tfrac{1}{4}p^{1/2}$.\\
\\
For the remaining six terms: using $a_i \geq \tfrac{3}{4} a_{i-1}$ gives $a_1 \leq \tfrac{4}{3\sqrt{3}}p^{1/4}$, $a_2 \leq \tfrac{1}{\sqrt{3}}p^{1/4}$ which bounds $\sqrt{a_1a_2}$. For $\sqrt{a_1a_3}$ we note that $a_2 \geq \tfrac{3}{4}a_1$, $a_4 \geq \tfrac{3}{4}a_3$ and therefore $\tfrac{p}{16} \geq \tfrac{9}{16}(a_1a_3)^2$. For $\sqrt{a_2a_3}$ we note that the maximum value of $a_2a_3$ occurs when $a_1=1$, $a_2 = \tfrac{2^{2/3}}{3}p^{1/3}$, $a_3=2^{-4/3}p^{1/3}$ and $a_4 = 3\cdot 2^{-10/3}p^{1/3}$. This yields $\sqrt{a_2a_3} \leq \tfrac{1}{2^{1/3}3^{1/2}}p^{1/3}$. We have $a_{1} a_{4} = \frac{p}{16 a_{2} a_{3}}$ and $a_{2} a_{3}$ is minimized when $a_{2} = 3/4$ and $a_{3} = 9/16$. This
gives $\sqrt{a_{1} a_{4}} \leq \frac{2}{3 \sqrt{3}} p^{1/2}$. Similarly,
$\sqrt{a_{2} a_{4}} \leq \frac{\sqrt{p}}{\sqrt{16 a_{1} a_{3}}} \leq \frac{1}{3} \sqrt{p}$ and $\sqrt{a_{3} a_{4}} \leq \frac{\sqrt{p}}{\sqrt{16 a_{1} a_{2}}} \leq \frac{1}{2 \sqrt{3}} \sqrt{p}$. Finally note that
$\sqrt{a_{4}} \leq \frac{2}{3 \sqrt{3}} p^{1/2}$. These substitutions give:
\begin{eqnarray*}
\displaystyle\sum_{n \leq x}r_{Q^\ast}(n) & \leq &256 \sqrt{a_1a_2a_3a_4}x^2p^{-2} + 64(\sqrt{a_1a_2a_3}+\sqrt{a_1a_2a_4}+\sqrt{a_1a_3a_4}+\sqrt{a_2a_3a_4})x^{3/2}p^{-3/2}\\
&&+ 16(\sqrt{a_1a_2}+\sqrt{a_1a_3}+\sqrt{a_1a_4}+\sqrt{a_2a_3}+\sqrt{a_2a_4}+\sqrt{a_3a_4}) xp^{-1}\\
&&+ 4(\sqrt{a_{1}} + \sqrt{a_{2}} + \sqrt{a_{3}} + \sqrt{a_{4}}) \sqrt{x} p^{-1/2} +1\\
& \leq & 64x^2p^{-3/2} + 64 \left(\frac{1}{3^{3/4}}p^{3/8}+ \frac{1}{3}p^{1/2}+\frac{1}{2\sqrt{3}}p^{1/2}+\frac{1}{4}p^{1/2} \right)x^{3/2}p^{-3/2}\\
&&+ 16 \left(\frac{2}{3}p^{1/4}+ \frac{1}{\sqrt{3}}p^{1/4}+ \frac{1}{2^{1/3}3^{1/2}}p^{1/3} + \frac{2}{3 \sqrt{3}} p^{1/2} + \frac{1}{3} p^{1/2} + \frac{1}{2 \sqrt{3}} p^{1/2}\right)x p^{-1}\\
&&+ 4 \left(\frac{2}{3^{3/4}} p^{1/8} + \frac{2^{1/3}}{\sqrt{3}} p^{1/6}
+ \frac{1}{\sqrt{3}} p^{1/4} + \frac{2}{3 \sqrt{3}} p^{1/2}\right) x p^{-1/2} + 1\\
\end{eqnarray*}
and the result follows.
\end{proof}
\begin{corollary}
\label{sumcor}  
For $1 \leq x \leq p$ we have $$\displaystyle\sum_{n \leq x} r_{Q^\ast}(n) \leq 179.12 \sqrt{x}$$ and for $x \geq p$ we have $$\displaystyle\sum_{n \leq x}r_{Q^{\ast}}(n) \leq 178.37x^2p^{-3/2}.$$
\end{corollary}
\begin{lemma}
\label{cusppart}
Let $p \geq 101$. Then
\[
  \dfrac{p}{p+1} \displaystyle\sum_{n=1}^\infty \dfrac{2^{\omega(\gcd(n,p))} \cdot 2r_{Q^\ast}(n)^2}{n} \displaystyle\sum_{d=1}^\infty \psi \left( d\sqrt{\dfrac{n}{p}} \right) \leq \frac{1}{\min Q^{*}} + 3216.6524 \frac{M(25.09p^{35/6})}{p^{1/4}}.
\]
\end{lemma}
\begin{proof}
For this proof, we proceed by cases on the size of $n$ relative to $p$. 
In many intervals we use the bound $\sum_{n \leq x}r_{Q^\ast}(n)^2 \leq r_{Q^\ast}(x)\left( \sum_{n \leq x}r_{Q^\ast}(n)\right)$; this then will use Lemmas $10$ and $11$.
\begin{itemize}
\item $1 \leq n \leq \frac{3}{4} p^{1/2}$:\\
  First, observe that if $r_{Q^{\ast}}(n) = 0$ for $1 \leq n \leq \frac{3}{4} p^{1/2}$, then the contribution from this segment is zero. Assume therefore that $\min Q^{\ast} \leq \frac{3}{4} p^{1/2}$.

If $\vec{x} = \begin{bmatrix} x_{1} \\ x_{2} \\ x_{3} \\ x_{4} \end{bmatrix} \ne \vec{0}$ and $i$ is the smallest positive integer so that $x_{i} \ne 0$,
then $Q^{\ast}(\vec{x}) \geq a_{i}^{*} x_{i}^{2} \geq a_{i}^{*}$. Hence
$\min Q^{\ast} \geq \min \{ a_{1}^{*}, a_{2}^{*}, a_{3}^{*}, a_{4}^{*} \}$.

We have that $a_{1}^{*} \geq (3/4) a_{2}^{*} \geq (9/16) a_{3}^{*} \geq (27/64) a_{4}^{*}$. This implies that
\[
\frac{p^{3}}{16} = a_{1}^{*} a_{2}^{*} a_{3}^{*} a_{4}^{*} \leq
\left(\frac{4}{3}\right)^{6} (a_{1}^{*})^{4}
\]
and so $a_{1}^{*} \geq \frac{9}{4^{8/3}} p^{3/4} > \frac{3}{4} p^{1/2}$.

We have $a_{1}^{*} = \frac{p}{4a_{1}} \leq p/4$ and so $p^{3}/16 = a_{1}^{*} a_{2}^{*} a_{3}^{*} a_{4}^{*}$ gives
\[
  \frac{p^{2}}{4} = \frac{p^{3}}{16 (p/4)} \leq \frac{p^{3}}{16 a_{1}^{*}} = a_{2}^{*} a_{3}^{*} a_{4}^{*} \leq \left(\frac{4}{3}\right)^{3} (a_{2}^{*})^{3}
\]
This gives $a_{2}^{*} \geq \frac{3p^{2/3}}{4^{4/3}} > \frac{3}{4} p^{1/2}$.

Similarly, $a_{2}^{*} = \frac{p}{4a_{2}} \leq p/3$ and so 
\[
\frac{3}{4p} = \frac{p^{3}}{16 (p^{2}/12)} \leq
\frac{p^{3}}{16 a_{1}^{*} a_{2}^{*}} = a_{3}^{*} a_{4}^{*} \leq \frac{4}{3} (a_{3}^{*})^{2}
\]
gives $a_{3}^{*} \geq \frac{3}{4} \sqrt{p}$. Since $\min Q^{*} < \frac{3}{4} \sqrt{p}$, it follows that $a_{4}^{*} < \frac{3}{4} \sqrt{p}$ and if
$Q^{*}(\vec{x}) < \frac{3}{4} \sqrt{p}$, then $x_{1} = x_{2} = x_{3} = 0$.
So if $r_{Q^{*}}(n) \ne 0$, then $n = a_{4}^{*} i^{2}$ and $r_{Q^{*}}(n) = 2$.
It follows that the contribution from this range is bounded by
\begin{align*}
  \frac{p}{p+1} \sum_{n=1}^{\frac{3}{4} p^{1/2}} \frac{2^{\omega(\gcd(n,p))} 2r_{Q^{*}}(n)^{2}}{n} \sum_{d=1}^{\infty} \psi\left(d \sqrt{\frac{n}{p}}\right)
  &\leq \frac{p}{p+1} \sum_{i=1}^{\left\lfloor \sqrt{\frac{3 \sqrt{p}}{4 a_{4}^{*}}} \right\rfloor}
  \frac{2 \cdot 2^{2}}{a_{4}^{*} i^{2}} \cdot \frac{3}{4 \pi^{2}}\\
  &\leq 8 \cdot \frac{3}{4 \pi^{2}} \sum_{i=1}^{\infty} \frac{1}{a_{4}^{*} i^{2}}\\
  &\leq \frac{1}{a_{4}^{*}} \cdot \frac{6}{\pi^{2}} \sum_{i=1}^{\infty} \frac{1}{i^{2}} = \frac{1}{a_{4}^{*}} = \frac{1}{\min Q^{\ast}}.
\end{align*}

Next we use Lemma~\ref{rQbound2} to give a bound on $r_{Q^{\ast}}(n)$.

We start by giving bounds on $a_{1}^{*}$, $a_{2}^{*}$ and $a_{1}^{*} a_{2}^{*}$
in terms of $a_{4}^{*}$. We have that $\frac{p^{3}}{16 a_{4}^{*}} = a_{1}^{*} a_{2}^{*} a_{3}^{*} \leq \left(\frac{3}{4}\right)^{3} (a_{1}^{*})^{3}$ and so
$a_{1}^{*} \geq \frac{4}{3 \cdot 2^{4/3}} p (a_{4}^{*})^{-1/3}$. We have
\[
  \frac{p^{3}}{16 (p/4) a_{4}^{*}}\leq \frac{p^{3}}{16 a_{1}^{*} a_{4}^{*}} = a_{2}^{*} a_{3}^{*} \leq \frac{3}{4} (a_{2}^{*})^{2}.
\] 
This gives $a_{2}^{*} \geq \frac{1}{\sqrt{3}} p (a_{4}^{*})^{-1/2}$.
We next derive an upper bound on $a_{3}^{*}$. We have
\[
  \frac{3^{3}}{4^{3}} (a_{3}^{*})^{3} a_{4}^{*} \leq a_{1}^{*} a_{2}^{*} a_{3}^{*} a_{4}^{*} = \frac{p^{3}}{16} 
\]
and so $a_{3}^{*} \leq \frac{2^{4/3}}{3} p (a_{4}^{*})^{-1/3}$. This gives
\[
  a_{1}^{*} a_{2}^{*} \geq \frac{p^{3}}{16 a_{3}^{*} a_{4}^{*}} \geq \frac{3p^{2}}{2^{16/3} (a_{4}^{*})^{2/3}}.
\]
Plugging this in gives
\begin{equation}
\label{individualbound}  
r_{Q^{*}}(n) \leq \left(29.328 \frac{n}{p} (a_{4}^{*})^{1/3}
+ 10.764 \sqrt{\frac{n}{p}} (a_{4}^{*})^{1/4} + 2\right) M(25.09np^{29/6}).
\end{equation}

\item {$\frac{3}{4} p^{1/2} \leq n \leq p/100$}:
In this, and remaining regions, we use that
\begin{align*}
  & \frac{p}{p+1} \sum_{n=\alpha}^{\beta} \frac{2^{\omega(\gcd(n,p))} 2r_{Q^{*}}(n)^{2}}{n} \sum_{d=1}^{\infty} \psi\left(d \sqrt{\frac{n}{p}}\right)\\
  &\leq
2 \left(\sum_{d=1}^{\infty} \psi\left(d \sqrt{\frac{\alpha}{p}}\right)\right) \cdot \left[\int_{\alpha}^{\beta} \frac{1}{t^{2}} \sum_{\alpha \leq n \leq t} 2^{\omega(\gcd(n,p))} r_{Q^{*}}(n)^{2} \, dt
+ \frac{1}{\beta} \sum_{\alpha \leq n \leq \beta} 2^{\omega(\gcd(n,p))} r_{Q^{*}}(n)^{2}\right]\\
\end{align*}
We use that $\gcd(n,p) = 1$ if $n < p$ and we apply the above result using equation \eqref{individualbound} to bound $r_{Q^{*}}(n)$
and $\sum_{n \leq x} r_{Q^{*}}(n) \leq 179.12 \sqrt{x}$. For $\alpha = \frac{3}{4} \sqrt{p}$ and $\beta = \frac{1}{100} p$, we get an upper bound of
\[
  \frac{1}{\sqrt{p}} \left[239.52 \left(a_{4}^{*}\right)^{1/3} + 146.52 \left(a_{4}^{*}\right)^{1/4} \ln(p) + 125.74 p^{1/4}\right] M(25.09p^{35/6}).
\]
\item $p/100 \leq n \leq p-1$: We use the same formulas as above and take $(\alpha,\beta) = (2^{k}p/100,2^{k+1}p/100)$ for $0 \leq k \leq 5$, as well as $(\alpha,\beta) = (64p/100,p)$. In total we obtain
\[
  \frac{1}{\sqrt{p}} \left[596.82 \left(a_{4}^{*}\right)^{1/3} + 955.35 \left(a_{4}^{*}\right)^{1/4} + 949.86\right].
\]

\item $n \geq p$:
Using that $a_{4}^{*} \geq 2$, we have that $r_{Q^{*}}(n) \leq 41.1 \left(\frac{n}{p}\right) (a_{4}^{*})^{1/3} M(25.09np^{29/6})$ for $n \geq p$. 
We have
\[
\frac{p}{p+1} \sum_{k=1}^{\infty} \sum_{n=k^{2} p}^{(k+1)^{2} p}
\frac{2 \cdot 2^{\omega(\gcd(n,p))} r_{Q^{*}}(n)^{2}}{n} \sum_{d=1}^{\infty} \psi\left(d \sqrt{\frac{n}{p}}\right).
\]
We use $\sum_{d=1}^{\infty} \psi(dx) \leq 9x^{3/2} e^{-4 \pi x}$ and get the bound
\begin{align*}
  & \sum_{k=1}^{\infty} 36 k^{3/2} e^{-4 \pi k} \sum_{n=k^{2} p}^{(k+1)^{2} p} \frac{r_{Q^{*}}(n)^{2}}{n}\\
  &= 36 \sum_{k=1}^{\infty} k^{3/2} e^{-4 \pi k}
  \left[\int_{k^{2} p}^{(k+1)^{2} p} \frac{1}{t^{2}} \left(\sum_{n \leq t}
    r_{Q^{*}}(n) \right) \left(\max_{n \leq t} r_{Q^{*}}(n)\right) \, dt\right.\\
  &\left.+ \frac{1}{(k+1)^{2} p} \left(\sum_{n \leq (k+1)^{2} p} r_{Q^{*}}(n)\right)
    \left(\max_{n \leq (k+1)^{2} p} r_{Q^{*}}(n)\right)\right].
\end{align*}
Using that $\sum_{n \leq x} r_{Q^{*}}(n) \leq 178.37 x^{2} p^{-3/2}$ gives the bound
$3665.51 \frac{(a_{4}^{*})^{1/3}}{\sqrt{p}} (4k^{3} + 6k^{2} + 4k + 1) M(25.09 (k+1)^{2} p^{35/6})$ on the integral.
The second term is bounded by $7331.1 (k+1)^{4} \frac{(a_{4}^{*})^{1/3}}{\sqrt{p}} M(25.09 (k+1)^{2} p^{35/6})$, for a total of
\[
\frac{(a_{4}^{*})^{1/3}}{\sqrt{p}} \sum_{k=1}^{\infty} k^{3/2} (131958.36 (4k^{3} + 6k^{2} + 4k + 1) + 263919.6 (k+1)^{4}) M(25.09 (k+1)^{2} p^{35/6}) e^{-4 \pi k}.
\]
If $n \geq 2$, there is an even integer $r$ with $1 \leq r \leq n$ for which
$\tau(r)$ is a maximum. (If the $r$ was odd, we could replace
an odd prime factor of $r$ with $2$ resulting in a lower number with
the same number of divisors.) This implies that for $n \geq 1$,
$M(2n) = \max_{r \leq n} \tau(2r) \leq \max_{r \leq n} \tau(2) \tau(r) =
\tau(2) M(n) = 2M(n)$. A straightforward induction then shows that
$M(kn) \leq 2k M(n)$. Applying this bound in the above formula gives
yields $\frac{(a_{4}^{*})^{1/3}}{\sqrt{p}} M(25.09 p^{35/6})$ times
a convergent series in $k$, whose sum is $\leq 82.525$ and so we get the contribution $\frac{82.525 (a_{4}^{*})^{1/3} M(25.09 p^{35/6})}{\sqrt{p}}$ to
the inner product from the portion with $n \geq p$.
\end{itemize}
Using the bound $a_{4}^{*} \leq \frac{4p^{3/4}}{3 \sqrt{3}}$ we obtain the bound
\[
\frac{1}{\min Q^{*}} + \frac{M(25.09p^{35/6})}{p^{1/4}}
\left(918.87 \sqrt[3]{\frac{4}{3 \sqrt{3}}} + 353.53 \sqrt[4]{\frac{4}{3 \sqrt{3}}} \cdot \frac{\log(p)}{p^{1/16}} + 125.74 + \frac{949.86}{p^{1/4}}\right).
\]
The maximum value of $\frac{\log(x)}{x^{1/16}}$ occurs when $x = e^{16}$.
Using the bound $\frac{\log(p)}{p^{1/16}} \leq \frac{16}{e}$ gives the desired result.
\end{proof}

Finally, we combine Lemma~\ref{eispart} and Lemma~\ref{cusppart} to prove Theorem~\ref{petbound}.
\begin{proof}[Proof of Theorem~\ref{petbound}]
  We simplify the formula for the sum of the right hand side of Lemma~\ref{eispart} and Lemma~\ref{cusppart} by allowing both terms to include $M(25.09p^{35/6})$. If $p = 101$, $25.09 \cdot p^{35/6} \approx 12341710124278$. We use a tabulated list of record high values of $\tau(n)$ (see \href{http://oeis.org/A002182}{OEIS sequence A002182}) to determine that the integer with the greatest number of divisors less than $25\cdot 101^{35/6}$ is $9316358251200$ and so $M(25.09 \cdot 101^{35/6}) = 10752$.

Next, we note that
\[
\frac{337.26 \log(p+2)}{p} \leq 0.00457 \frac{10752}{p^{1/4}} \leq
0.00457 \frac{M(25.09 p^{35/6})}{p^{1/4}}
\]
for $p \geq 101$. This can be seen using the fact that $\log(x+2)/x^{3/4}$ is decreasing for $x > e^{4/3} - 2$. Finally,
\[
\frac{206.67}{p} \leq 0.00061 \frac{10752}{p^{1/4}} \leq 0.00061 \frac{M(25.09 p^{35/6})}{p^{1/4}}
\]
for $p \geq 101$. Therefore
\[
\langle C, C \rangle \leq \frac{1}{\min Q^{*}} +
(3216.6524 + 0.00457 + 0.00061) \frac{M(25.09p^{35/6})}{p^{1/4}}
\leq \frac{1}{\min Q^{*}} + 3216.66 \frac{M(25.09p^{35/6})}{p^{1/4}}
\]
as desired.
\end{proof}

\begin{proof}[Proof of Theorem~\ref{ineq}]
  We have $r_{Q}(n) = a_{E}(n) + a_{C}(n)$. A lower bound on $a_{E}(n)$ was
  given in Theorem~\ref{eisineq}. Using the bound $C_{Q} \leq \sqrt{\frac{As}{B}}$, where $A$ is the bound given in Theorem~\ref{petbound} and
  $B = \frac{3}{208 \pi^{4}} \cdot \frac{p}{(p+1) \log(p)}$ we obtain
\[
  C_{Q} \leq \left(\sqrt{\frac{1}{12} \cdot \left(1 + \frac{1}{p}\right) \cdot \frac{208 \pi^{4}}{3}}\right) \sqrt{\frac{1}{\min Q^{*}} + 3216.66 \frac{M(25.09p^{35/6})}{p^{1/4}}}.
  \]
We use that for $p \geq 101$,
\[
\left(\sqrt{\frac{1}{12} \cdot \left(1 + \frac{1}{p}\right) \cdot \frac{208 \pi^{4}}{3}}\right) \leq 23.85.
\]
\end{proof}

%%%%%%%%%
%%%%%%%%%%
\section{Sum of Exceptions}

In this section, we prove Corollary \ref{sumofexcep}. We proceed by using Theorem \ref{petbound}
 and relating $\langle C, C \rangle$ to an integral of the form
\[
  \int_{\sigma}^{\infty} \int_{-1/2}^{1/2} |C(x+iy)|^{2} \, dx \, dy.
\]

\begin{lemma}
\label{multbound}
Let $\sigma > 0$ and let $F(\sigma)$ be the largest number of points in the region $\{ x+iy : -1/2 \leq x \leq 1/2, y \geq \sigma \}$ that can be in a single
$\Gamma_{0}(p)$ orbit. Then, $F(\sigma) \leq 1 + \frac{\sqrt{3}}{2 p \sigma}$.
\end{lemma}
\begin{proof}
Choose a $\Gamma_{0}(p)$-orbit and let $w = \alpha + i \beta$ be a point with $|\alpha| \leq 1/2$, $\beta \geq \sigma$ that has lowest imaginary part (subject to $\beta \geq \sigma$) in this orbit.  We will count the number of matrices $M = \begin{bmatrix} a & b \\ c & d \end{bmatrix}$ (up to negation) in $\Gamma_{0}(p)$ with ${\rm Im}(M(w)) \geq {\rm Im}(w)$ and $|{\rm Re}(M(w))| \leq 1/2$. %If $c = 0$ we
%take $d = 1$. Otherwise we assume that $c > 0$.

We have
\[
  {\rm Im}(M(w)) = \frac{{\rm Im}(w)}{|cw+d|^{2}}.
\]
So in order for ${\rm Im}(M(w)) \geq w$, we need $|cw+d|^{2} \leq 1$.
If $c = 0$, then in order for $M$ to be in $\SL_{2}(\Z)$ we must have $d = \pm 1$ and so up to negation we have $d = 1$. This gives us one choice of $(c,d)$.

Assume then that $c > 0$. We have $|cw+d|^{2} = (\alpha+d)^{2} + (c\beta)^{2} \leq 1$
and also $\gcd(c,d) = 1$. Since $-1/2 \leq \alpha \leq 1/2$,
there is at most one choice of $d$ so that $(\alpha+d)^{2} \leq 1$
(either $d = 1$ or $d = -1$). In any case, $|\alpha+d| \geq 1/2$
and so $(\alpha+d)^{2} \geq 1/4$. Thus, $(c \beta)^{2} \leq 3/4$
and so $c \leq \frac{\sqrt{3}}{2 \beta}$. Recalling that $c$ must be a multiple
of $p$, the total number of choices of $(c,d)$ with $c > 0$
is at most $\frac{\sqrt{3}}{2p \beta}$.

Finally, for each choice of $(c,d)$, if $M_{1} = \begin{bmatrix}
a & b \\ c & d \end{bmatrix}$ and $M_{2} = \begin{bmatrix}
a' & b' \\ c & d \end{bmatrix}$ are two different matrices in $\Gamma_{0}(p)$
with the same choice of $(c,d)$, then
$M_{1} M_{2}^{-1} = \begin{bmatrix} * & * \\ 0 & 1 \end{bmatrix}$. %\textcolor{red}
If $-1/2 \leq {\rm Re}(M_{2}(w)) \leq 1/2$, then
${\rm Re}(M_{1}(w)) = {\rm Re}(M_{1} M_{2}^{-1} M_{2}(w))$ is
not between $-1/2$ and $1/2$. Thus, for each $(c,d)$ pair, there is
at most one choice of $a, b$ that makes $-1/2 \leq {\rm Re}(M(w)) \leq 1/2$.
This proves the claim. 
\end{proof}

It follows from the claim that
\[
\frac{1}{p+1} \int_{\sigma}^{\infty} \int_{-1/2}^{1/2} |C(x+iy)|^{2} \, dx \, dy
\leq \left(1 + \frac{\sqrt{3}}{2 p \sigma}\right) \langle C, C \rangle
\]
since the region $\sigma \leq y \leq \infty$ is contained in at most
$F(\sigma) \leq 1 + \frac{\sqrt{3}}{2 p \sigma}$ different fundamental domains
for $\Gamma_{0}(p)$.

A straightforward calculation shows that
\begin{align*}
\frac{1}{p+1} \int_{\sigma}^{\infty} \int_{-1/2}^{1/2} |C(x+iy)|^{2} \, dx \, dy
&= \int_{\sigma}^{\infty} \int_{-1/2}^{1/2} \sum_{m,n} a_{C}(m) a_{C}(n) e^{-2 \pi (m+n)y} e^{2 \pi i (m-n)x} \, dx \, dy\\
  &= \int_{\sigma}^{\infty} \sum_{n=1}^{\infty} a_{C}(n)^{2} e^{-4 \pi n y} \, dy\\
  &= \frac{1}{4 \pi} \sum_{n=1}^{\infty} \frac{a_{C}(n)^{2}}{n} e^{-4 \pi n \sigma}.
\end{align*}
If $n$ is an integer so that $r_{Q}(n) = 0$, then $a_{C}(n) = -a_{E}(n)$.
We have that $a_{E}(n) \geq \frac{24 (p-1)}{p^{3/2}} \phi(n) \gg \frac{1}{\sqrt{p}} \frac{n}{\log \log(n)}$ if $n > 2$.

Let $m$ be the largest positive integer $n$ so that $r_{Q}(n) = 0$ and $\sigma = 1/m$. Theorem~\ref{ineq}
shows that $m \ll \max \{ \frac{p^{2 + \epsilon}}{\min Q^{*}}, p^{7/4 + \epsilon} \}$.
  
It follows that
\begin{align*}
  \sum_{r_{Q}(n) = 0} n &\ll 3 + (\log \log m)^{2} \sum_{r_{Q}(n) = 0} \frac{pn}{(\log \log(n))^{2}}\\
  &\ll 3 + p (\log \log m)^{2} \sum_{r_{Q}(n) = 0} \frac{a_{E}(n)^{2}}{n}\\
  &\ll 3 + p (\log \log m)^{2} \sum_{r_{Q}(n) = 0} \frac{a_{C}(n)^{2}}{n}\\
  &\ll 3 + p (\log \log m)^{2} \sum_{n=1}^{\infty} \frac{a_{C}(n)^{2}}{n} e^{-4 \pi n \sigma}\\
  &\ll 3 + (\log \log m)^{2} (p+1) p F(1/m) \langle C, C \rangle\\
  &\ll p^{2 + \epsilon} \left(1 + \frac{\sqrt{3} m}{2p}\right) \left(\frac{1}{\min Q^{*}} + p^{-1/4 + \epsilon}\right).
\end{align*}
If $\min Q^{*} \leq p^{1/4 - \epsilon}$, we get that
\[
\sum_{r_{Q}(n) = 0} n \ll p^{1+\epsilon} \frac{m}{\min Q^{*}} \ll
\frac{p^{3+\epsilon}}{\left(\min Q^{*}\right)^{2}}.
\]
If $\min Q^{*} > p^{1/4 - \epsilon}$, we get
\[
\sum_{r_{Q}(n) = 0} n \ll p^{1+\epsilon} m p^{-1/4 + \epsilon} \ll p^{5/2 + \epsilon}.
\]
In all cases, we obtain
\[
  \sum_{r_{Q}(n) = 0} n \ll \max \left\{ \frac{p^{3+\epsilon}}{(\min Q^{*})^{2}}, p^{5/2 + \epsilon} \right\}.
\]
as desired.

%%%%%%%%%
%%%%%%%%%
\section{Family of Forms With Explicit Exceptions}
We end this paper with an example of a family of prime discriminant forms, and the classification of the excepted values. 
\begin{theorem}
\label{Family}
Let $p \equiv 5 \pmod{8}$ be prime. Consider the quadratic form $$Q_p(\vec{x}) = x^2+xy+xz+xw+y^2+yz+yw+z^2+zw+ \frac{p+3}{8}w^2.$$ This quaternary form is almost universal; more specifically, it represents all $n \in \mathbb N$ except those $n<\tfrac{p}{8}$ of the form $n=4^k(16 \ell+14)$ for integers $k$ and $\ell$.
\end{theorem}
We begin with a pair of lemmas regarding the integers represented by a particular ternary subform of $Q_p$ and by a specific ternary form. We begin by discussing the integers represented by the subform $Q_p((x,y,z,0))$.
\begin{lemma}
\label{Sub}
The ternary quadratic form $x^2+xy+xz+y^2+yz+z^2$ is regular and of class number one. It represents all positive integers except those of the form $4^k(16 \ell + 14)$. 
\end{lemma}
\begin{proof}%[Proof of Lemma \ref{Sub}:]
That the form is regular and of class number one is readily found in \cite{JKS}. In particular, this means that $n \in \mathbb N$ is represented by the form if and only if $n$ is locally represented by the form. Clearly, over $\mathbb R$ all $n \in \mathbb N$ are represented. Over $\mathbb Q_p$ for $p \neq 2$, Hensel's Lemma states it is sufficient to check $\pmod{p}$ to guarantee representation. As this form is nondegenerate and of dimension at least $2$, that makes it universal over a finite field, and specifically $\pmod{p}$. The only prime left to consider is $p=2$. Using Siegel's theory of local densities, one must check whether any quotients of an integer by a square factor are represented $\pmod{2^5}$. A simple Sage computation shows that $n \equiv 14, 30 \pmod{32}$ are not locally-represented, which implies that $n = 4^k(16 \ell+14)$ are not represented. Moreover, for all other $n$, there are Good-type solutions (using the language of \cite{Hanke}) which means $n$ is represented over $\mathbb Q_2$.
\end{proof}
We also will reference the following result:
\begin{lemma}
\label{Ternary}
The ternary form $X^2+2Y^2+6Z^2$ is regular, and represents every positive integer $N \equiv 9 \pmod{24}$. Additionally, there is a representation of the form $X=4z+1$, $Y=3y+z+1$, and $Z = 2x+y+z+1$ for some integers $x,y,z$.
\end{lemma}
\begin{proof}%[Proof of Lemma \ref{Ternary}:]
That the form is regular again can be found in \cite{JKS}. Clearly, over $\mathbb R$ all $n \equiv 9 \pmod{24}$ are represented. For $p \neq 2,3$, over $\mathbb Q_p$ Hensel's Lemma again states it is sufficient to check $\pmod{p}$, and again in those cases the form is ternary and nondegenerate so it is actually universal. For $p=3$, note that $(1,2,0)$ is a solution with at least one nonzero partial derivative, so it can be lifted to a solution in $\mathbb Z_3$. For $p=2$, all partial derivatives vanish, and so it suffices to look $\pmod{32}$ to ensure a $2$-adic solution. Running Sage \cite{Sage}, we see that there are solutions. By local-global principles, then, $X^2+2Y^2+6Z^2$ represents all positive integers $\equiv 9 \pmod{24}$.\\
\\
Now set $N = 24m+9$ for some nonnegative integer $m$. For $24m+9=X^2+2Y^2+6Z^2$, $X$ must be odd. Without loss of generality by swapping $X$ and $-X$ as necessary, $X=4z+1$ for $z \in \mathbb Z$. Considering now the equation $\pmod{3}$, one has $X \equiv \pm Y \pmod{3}$. Again, without loss of generality suppose $X \equiv Y \pmod{3}$. Writing $X-Y = 3k$ for some integer $k$ and solving for $Y$ yields $Y=3(k+z)+z+1$. Last, consider the equation $\pmod{8}$. Then $Y$ and $Z$ must have the same parity. Therefore $Z-Y =2 \ell$ for $\ell \in \mathbb Z$. Substituting again, this means $Z=2(\ell+k+z)+(k+z)+z+1$. And noting the parentheses, letting $y=k+z$ and $x=\ell+k+z$, the claim holds. 
\end{proof}
\begin{proof}[Proof of Theorem \ref{Family}:]
By Lemma \ref{Sub}, we need only consider representation of integers $n=4^k(16 \ell+14)$ by the form $Q_p$. We begin by diagonalizing the form over $\mathbb Q$:
\begin{eqnarray*}
Q_p(\vec{x}) & = & x^2+xy+xz+xw+y^2+yz+yw+z^2+zw+ \frac{p+3}{8}w^2\\
& = & \left( x+ \frac{1}{2}(y+z+w)\right)^2 + \frac{3}{4} \left(y+\frac{1}{3}(z+w) \right)^2 + \frac{2}{3} \left( z+ \frac{1}{4}w \right)^2 + \frac{p}{8}w^2.
\end{eqnarray*}
If $Q_p$ represents $n$, then $\frac{p}{8}w^2 \leq n$, or $\vert w \vert \leq 2 \left( \sqrt{\frac{2n}{p}} \right)$. When $\sqrt{\frac{2n}{p}}<\frac{1}{2}$, or equivalently when $n< \frac{p}{8}$, this forces $w=0$ and Lemma \ref{Sub} shows such $n$ cannot be represented by the subsequent ternary subform.\\
\\
Next suppose $n=4^k(16\ell+14)$ with $n \geq \frac{p}{8}$. Set $m = n-\frac{p+3}{8}$. By Lemma \ref{Ternary} there exist integers $x,y,z$ with 
\begin{eqnarray*}
24m+9 & = & (4z+1)^2+2(3y+z+1)^2+6(2x+y+z+1)^2.
\end{eqnarray*}
Expanding and simplifying, this gives
\begin{eqnarray*}
n &= &x^2+xy+xz+y^2++yz+z^2+z+y+z+\frac{p+3}{8} \\
& = & Q_p((x,y,z,1)).
\end{eqnarray*}
Hence, $Q_p$ represents $n$.
\end{proof}

\end{document}